\documentclass[11pt,reqno]{amsart}
\usepackage[margin=1.15in]{geometry}
\usepackage{amsmath,amssymb,amsthm,mathtools,booktabs,array}
\usepackage{enumitem}
\usepackage{hyperref}
\hypersetup{colorlinks=true,linkcolor=blue,citecolor=blue,urlcolor=blue}

\DeclareMathOperator{\sn}{sn}\DeclareMathOperator{\cn}{cn}
\DeclareMathOperator{\dn}{dn}
\DeclareMathOperator{\am}{am}
\newcommand{\dd}{\,\mathrm{d}}
\newcommand{\RR}{\mathbb{R}}\newcommand{\ZZ}{\mathbb{Z}}
\newcommand{\PP}{\ensuremath{\mathrm{P}}}
\newcommand{\DD}{\ensuremath{\mathrm{D}}}
\newcommand{\Lat}{\mathcal{L}}

\theoremstyle{plain}
\newtheorem{theorem}{Theorem}[section]
\newtheorem{proposition}[theorem]{Proposition}

\newtheorem{lemma}[theorem]{Lemma}
\theoremstyle{definition}
\newtheorem{remark}[theorem]{Remark}
\newtheorem{problem}[theorem]{Problem}
\numberwithin{equation}{section}

\title[The two cubic separable minimal surfaces]
{The two cubic separable minimal surfaces:\\
$\;\mathcal D_1\mathcal D_2\mathcal D_3=1\;$ and $\;
\mathcal C_1+\mathcal C_2+\mathcal C_3=3\,\mathcal C_1\mathcal C_2\mathcal C_3$}
\author{Steven Finch}
\address{MIT Sloan School of Management, Cambridge, MA 02139, USA}
\email{steven\_finch\_math@outlook.com}
\date{\today}
\subjclass[2020]{Primary 53A10; Secondary 33E05, 53C42}
\keywords{Triply periodic, Schwarz D surface, Schwarz diamond surface, Schwarz P surface, Schwarz primitive surface, separable minimal surface, Jacobi elliptic functions, conjugate minimal surfaces}

\begin{document}
\begin{abstract}
A minimal surface in $\RR^{3}$ is \emph{separable} if it is the zero set
of $f(x)+g(y)+h(z)$, and \emph{isotropic} if moreover $f=g=h$. We
announce that there are exactly two non-planar isotropic separable
minimal surfaces up to homothety and rigid motion, that they are
conjugate, and that they are Schwarz's diamond surface $\DD$ and
Schwarz's primitive surface $\PP$. In lattice-normalized coordinates
their implicit equations are
\[
  \DD:\quad \mathcal D(x)\,\mathcal D(y)\,\mathcal D(z)=1,
  \qquad\qquad
  \PP:\quad \mathcal C(x)+\mathcal C(y)+\mathcal C(z)
           =3\,\mathcal C(x)\mathcal C(y)\mathcal C(z),
\]
where
$\mathcal D(t)=\frac{\sn\dn}{\cn}\bigl(K[\tfrac14]t,\tfrac14\bigr)$ and
$\mathcal C(t)=\cn\bigl(2K[\tfrac34]t,\tfrac34\bigr)$. The first
equation is classical --- it is due to Schwarz, later Cayley, and appears in
Nitsche's \emph{Lectures} --- we identify Nitsche's transcendental
function exactly as $\sn\dn/\cn$ at parameter $\tfrac14$. The second
appears to be new; it is the symmetric member of the two-parameter
family of Kim \& Ogata (2024), and it proves their assertion that the 
family contains $\PP$. Both surfaces are governed by one differential
equation with two signs,
\[
  \varphi'^{\,2}=1+2\cosh 2\varphi\ \ (\DD),
  \qquad
  \varphi'^{\,2}=1+2\cos 2\varphi\ \ (\PP),
\]
the two signs of the separation constant of the classical reduction;
and in that common normalization the two lattice periods are
$2K[\tfrac14]$ and $2K[\tfrac34]$, whose ratio is Schwarz's $1866$
constant $K'[\tfrac14]/K[\tfrac14]=1.2792615\ldots$, the necessary
scale ratio for a conjugate pair. Proofs will appear in
\cite{Dpaper,Ppaper}.
\end{abstract}
\maketitle

\section{The two equations}

Elliptic functions are in the \emph{parameter} convention $m=k^{2}$, as
in \cite{BF} and in Mathematica's \texttt{EllipticK},
\texttt{JacobiCN}, \texttt{JacobiAmplitude}. A surface
$\Sigma\subset\RR^{3}$ is \emph{separable} \cite{Wein,HL,KO} if
\[
  \Sigma=\{(x,y,z): f(x)+g(y)+h(z)=0\},\qquad
  \nabla(f+g+h)\neq0 \ \text{on}\ \Sigma,
\]
and \emph{isotropic} if $f$, $g$, $h$ agree up to additive constants.
Define
\begin{equation}\label{eq:DC}
  \mathcal D(t):=\frac{\sn\dn}{\cn}\Bigl(K[\tfrac14]\,t,\ \tfrac14\Bigr),
  \qquad
  \mathcal C(t):=\cn\Bigl(2K[\tfrac34]\,t,\ \tfrac34\Bigr),
\end{equation}
so that both have period $2$ and
\begin{equation}\label{eq:DCvalues}
\mathcal D(0)=0,\ \mathcal D(\tfrac12)=1,\
\mathcal D(1)=\infty,\ \mathcal D(t+1)=-\frac1{\mathcal D(t)};
\ \mathcal C(0)=1,\ \mathcal C(\tfrac12)=0,\
\mathcal C(1)=-1,\ \mathcal C(t+1)=-\mathcal C(t).
\end{equation}

Both are even about their natural centers; $\mathcal C$ is even,
bounded by $1$, and real-analytic on $\RR$, whereas $\mathcal D$ is odd
and has poles at the odd integers.

\begin{theorem}[the diamond surface; classical]\label{thm:D}
Let $\DD$ be Schwarz's diamond surface, normalized so that its
translation lattice is the face-centered cubic lattice generated by
$(1,1,0)$, $(1,0,1)$, $(0,1,1)$. Then, up to a lattice translation,
\begin{equation}\label{eq:Deq}
  \boxed{\ \mathcal D(x)\,\mathcal D(y)\,\mathcal D(z)=1\ }
\end{equation}
on the fundamental region $x,y,z\in(0,1)$; globally the fcc‑invariant form 
is $\Phi(x)\Phi(y)+1=\Phi(z)(\Phi(x)+\Phi(y))$ with $\Phi=|\mathcal F|$
and $\mathcal F$ defined by \eqref{eq:Dquad}.
\end{theorem}

Equation \eqref{eq:Deq} is \emph{not} new: it is equation \textup{(38)}
of \textup{\cite[\S83]{Nitsche}}, where the surface is attributed to
Schwarz \textup{\cite[vol.~I, pp.~1--125]{Schwarz}} and the
classification of the surrounding class to Cayley \textup{\cite{Cayley}};
see \S\ref{sec:credit}. What is new in Theorem \ref{thm:D} is only 
the closed form of the transcendental function: Nitsche defines his 
$\mathcal E$ implicitly, as
the inverse of $\xi=\int_{0}^{t}(1+\tau^{2}+\tau^{4})^{-1/2}\dd\tau$,
and we observe that
\begin{equation}\label{eq:NitscheE}
\mathcal E(u)=\frac{\sn\dn}{\cn}\Bigl(u,\tfrac14\Bigr).
\end{equation}
Both sides of \eqref{eq:NitscheE}, as functions of $u$, solve
$y'^{2}=1+y^{2}+y^{4}$ with $y(0)=0$, $y'(0)=1$: for the left side this
is the definition of $\mathcal E$, and for the right side it is the
identity
$\bigl(\tfrac{\sn\dn}{\cn}\bigr)'^{2}
=1+(2-4m)\bigl(\tfrac{\sn\dn}{\cn}\bigr)^{2}
+\bigl(\tfrac{\sn\dn}{\cn}\bigr)^{4}$ at $m=\tfrac14$.
Consistently, the first pole of $\mathcal E$ is the first zero of
$\cn(\cdot,\tfrac14)$: Nitsche's period constant
$p_{0}=\int_{0}^{\infty}(1+\tau^{2}+\tau^{4})^{-1/2}\dd\tau$, which he
evaluates as $\tfrac23K[\tfrac89]$ \textup{\cite[\S84]{Nitsche}}, is
$K[\tfrac14]$ by one descending Landen step. This is why the argument
of $\mathcal D$ in \eqref{eq:DC} is $K[\tfrac14]\,t$, and why the cube
edge in Table \ref{tab:par} is $2K[\tfrac14]=2p_{0}$.

\begin{theorem}[the primitive surface]\label{thm:P}
Let $\PP$ be Schwarz's primitive surface, normalized so that its
translation lattice is the body-centered cubic lattice
$\Lat$ generated by $2\ZZ^{3}$ and $(1,1,1)$, and so that the origin is
a center of one of the two labyrinths. Put
\begin{equation}\label{eq:Pfun}
  \mathcal P(x,y,z):=\mathcal C(x)+\mathcal C(y)+\mathcal C(z)
   -3\,\mathcal C(x)\,\mathcal C(y)\,\mathcal C(z).
\end{equation}
Then
\begin{equation}\label{eq:Peq}
  \boxed{\ \{\mathcal P=0\}=\PP\ \sqcup\ \Lat,\ }
\end{equation}
where the exceptional set is the lattice $\Lat$ itself,
$\Lat=\{n\in\ZZ^{3}:n_{1}\equiv n_{2}\equiv n_{3}\ (\mathrm{mod}\ 2)\}$,
a discrete set disjoint from the topological closure $\overline{\PP}$ of $\PP$. Thus \eqref{eq:Peq} determines $\PP$.
\end{theorem}

\begin{remark}[the exceptional set]\label{rem:exc}
$\Lat$ is unavoidable and completely understood. By
Proposition \ref{prop:sep} below,
$\mathcal P$ vanishes exactly where
$\Psi(x)+\Psi(y)+\Psi(z)\in\pi\ZZ$, with $\Psi$ the bounded function
\eqref{eq:Psi} below; the surface is the case $0$, and
the cases $\pm\pi$ force the equality
$\Psi(x)=\Psi(y)=\Psi(z)=\pm\frac\pi3$, i.e., $x,y,z\in\ZZ$ of equal
parity. These are the turning points of the pendulum
\eqref{eq:odeP} --- geometrically, the centers of the two labyrinths,
one lattice orbit each, of which \eqref{eq:Pfun} sees one. Replacing
$\mathcal C$ by $\cn\dn$ and $\sn$ data at parameter $-3$ removes
$\Lat$ at the cost of the trilinear shape.
\end{remark}

\begin{remark}[symmetry at sight]\label{rem:sym}
Equation \eqref{eq:Pfun} is invariant under all $6$ permutations of
$(x,y,z)$, under all $8$ sign changes $x\mapsto-x$ (because
$\mathcal C$ is even), under $2\ZZ^{3}$ (period $2$), and under
$(x,y,z)\mapsto(x+1,y+1,z+1)$ (because both sides of \eqref{eq:Peq}
are odd under $\mathcal C\mapsto-\mathcal C$). This is the full space
group $Im\bar3m$ of $\PP$, read off without computation, with no
absolute values anywhere: the equation is sign-honest, and
$\operatorname{sign}\mathcal P$ distinguishes the two labyrinths. By
contrast, the real fcc invariance of \eqref{eq:Deq} needs two
simultaneous shifts, since $\mathcal D(t+1)=-1/\mathcal D(t)$ carries a
sign; over $\RR$ this is the source of the absolute values in
\cite{Dpaper}.
\end{remark}

\begin{remark}[hand checks for $\PP$]\label{rem:hand}
In the coordinates of \cite{Ppaper} --- base point $(0,0,\frac12)$,
flat points $(\pm\frac12,\pm\frac12,\frac12)$ --- \eqref{eq:Peq} reads
$\mathcal C(x)-\mathcal C(y)-\mathcal C(z)
 =3\,\mathcal C(x)\mathcal C(y)\mathcal C(z)$,
the two forms being exchanged by $(x,y,z)\mapsto(x,1-y,1-z)$ together
with \eqref{eq:DCvalues}. The four segments joining
$(0,0,\frac12)\to(\frac12,-\frac12,\frac12)\to(\frac12,0,1)
\to(\frac12,\frac12,\frac12)\to(0,0,\frac12)$, four contiguous edges of
a regular octahedron spanning a skew rhombus, then satisfy the equation
identically \emph{by parity and antiperiodicity alone}, with no
evaluation of an elliptic function: on $\{(t,\pm t,\frac12)\}$ it is
$\mathcal C(t)\mp\mathcal C(t)-0=0$, and on $\{(\frac12,t,1-t)\}$ and
$\{(\frac12,t-\frac12,t+\frac12)\}$ it is
$0-\mathcal C(t)-\mathcal C(1-t)=0$. The analogous statement for $\DD$
is that its patch is spanned by four edges of a regular tetrahedron
\textup{\cite[\S84, fig.~7]{Nitsche}}. Equivalently, by \eqref{eq:Psi} the relation is
$\Psi(x)=\Psi(y)+\Psi(z)$ in those coordinates, and the four segments
are $\Psi(t)=\Psi(\pm t)+\Psi(\tfrac12)$ and
$\Psi(\tfrac12)=\Psi(t)+\Psi(1-t)$, both immediate from
\eqref{eq:Psivalues}.
\end{remark}

\section{One equation, two signs}\label{sec:one}

The parallel between \eqref{eq:Deq} and \eqref{eq:Peq} is not an
analogy but an identity of mechanism. Following Weingarten
\cite{Wein} and Nitsche \cite[\S82]{Nitsche}, write a separable minimal
surface as $u+v+w=0$ with $u=f(x)$, etc., set $X(u)=f'(x)^{2}$, and
recall that minimality is equivalent to
\begin{equation}\label{eq:mse}
  (Y+Z)X'+(Z+X)Y'+(X+Y)Z'=0 ,
\end{equation}
which forces $X'''/X'=Y'''/Y'=Z'''/Z'=\kappa$ for a constant $\kappa$,
the \emph{separation constant}.

\begin{theorem}[the dichotomy, and uniqueness]\label{thm:unique}
Up to homothety and rigid motion there are exactly two non-planar
isotropic separable minimal surfaces in $\RR^{3}$. In the normalization
$|\kappa|=4$ with unit coefficients, their common separating function
$\varphi$ satisfies
\begin{align}
  \kappa=+4:\qquad &\varphi'(t)^{2}=1+2\cosh\bigl(2\varphi(t)\bigr),
   \label{eq:odeD}\\
  \kappa=-4:\qquad &\varphi'(t)^{2}=1+2\cos\bigl(2\varphi(t)\bigr),
   \label{eq:odeP}
\end{align}
and the surfaces $\varphi(x)+\varphi(y)+\varphi(z)=0$ are
\begin{subequations}\label{eq:phis}
	\begin{align}
	\DD:\quad \varphi(t)&=\log\frac{\sn\dn}{\cn}\Bigl(t,\tfrac14\Bigr),
	\label{eq:phiD}\\
	\PP:\quad \varphi(t)&=\am\Bigl(\sqrt3\,t+a,\ \tfrac43\Bigr)
	=\arctan\Bigl(\sqrt3\,\cn\bigl(2t,\tfrac34\bigr)\Bigr),
	\ a:=F[\pi/3\,|\,4/3]=\tfrac{\sqrt3}{2}K[\tfrac34].
	\label{eq:phiP}
	\end{align}
\end{subequations}
There is no isotropic example with $\kappa=0$. The two surfaces are
conjugate: $\PP$ is the adjoint of $\DD$.
\end{theorem}
\noindent Here \eqref{eq:phiP} is in the ODE normalization of
Theorem \ref{thm:unique}, of antiperiod $K[\tfrac34]$; the
lattice-normalized function of \eqref{eq:DC} is
$\Psi(t):=\arctan\bigl(\sqrt3\,\mathcal C(t)\bigr)$, so that
$\varphi_{\PP}(t)=\Psi\bigl(t/K[\tfrac34]\bigr)$.
Three features of Theorem \ref{thm:unique} deserve emphasis.

\subsection*{(i) The parameters are forced, not chosen}
In the $\kappa>0$ branch, isotropy applied to Nitsche's coefficient
relations \cite[(33)]{Nitsche} gives $ab=c^{2}$ and $ac=b^{2}$, hence
$a=b=c$; normalizing to $1$ yields $X(u)=1+e^{2u}+e^{-2u}$, which is
\eqref{eq:odeD}, and the parameter $\tfrac14$ in \eqref{eq:phis}
follows. In the $\kappa<0$ branch, isotropy applied to
\cite[Lem.~4.1]{KO} gives $a_{i}=A$ and $|b_{i}|=A$ for a single
$A>0$, so the amplitude parameter is
$4\sqrt{|BC|}/E_{1}=4A/3A=\tfrac43$ \emph{for every $A$}: the
parameter $\tfrac43$ is forced and only the scale is free. This is
what identifies the member of the Kim--Ogata family, and proves
\cite[Thm.~1(3)]{KO}.

\subsection*{(ii) Bounded versus unbounded, and why $\PP$ is the tidier surface}
The right side of \eqref{eq:odeD} is $\ge3>0$, so $\varphi$ is a
diffeomorphism onto $\RR$: the associated function
$e^{\varphi}=\sn\dn/\cn$ is unbounded, has poles, and
$\varphi=\log(\cdot)$ is defined only where that function is positive.
The right side of \eqref{eq:odeP} vanishes at
$\varphi=\pm\frac\pi3$, so $\varphi$ \emph{oscillates} between the two
turning points. In the lattice normalization of \eqref{eq:DC} we write
\begin{equation}\label{eq:Psi}
\Psi(t):=\arctan\bigl(\sqrt3\,\mathcal C(t)\bigr)
=\am\bigl(a(1+2t),\ \tfrac43\bigr),
\qquad a=F[\pi/3\,|\,4/3],
\end{equation}
so that $\varphi_{\PP}(t)=\Psi\bigl(t/K[\tfrac34]\bigr)$: the function
$\varphi_{\PP}$ of \eqref{eq:phiP} has antiperiod $K[\tfrac34]$, whereas
$\Psi$ has antiperiod $1$. Then $\Psi$ is real-analytic and even on
$\RR$, with
\begin{equation}\label{eq:Psivalues}
\Psi(0)=\tfrac\pi3,\ \Psi(\tfrac14)=\tfrac\pi4,\
\Psi(\tfrac12)=0,\ \Psi(\tfrac34)=-\tfrac\pi4,\
\Psi(1)=-\tfrac\pi3,\ \Psi(t+1)=-\Psi(t),\ 
|\Psi|\le\tfrac\pi3 ,
\end{equation}
and $\Psi$ decreases strictly from $\frac\pi3$ to $-\frac\pi3$ on
$[0,1]$.
Every awkwardness in $\DD$ --- the absolute values, the projective
degenerations at poles, the need to argue in transformed variables ---
and every convenience in $\PP$ --- global separability, sign-honesty,
the trilinear normal form --- is this one difference. It is also the
reason the Jacobi amplitude in \eqref{eq:phis} must be read in its
\emph{bounded} branch, $\am(u,m)=\arcsin\sn(u,m)$ for $m>1$: at
parameter $\tfrac43>1$ the naive monotone reading is a different
function, and with it \eqref{eq:Peq} is false.

\subsection*{(iii) Schwarz's constant emerges}
In the common normalization of Theorem \ref{thm:unique} the two
solutions of \eqref{eq:odeD}, \eqref{eq:odeP} have $t$-periods
\begin{equation}\label{eq:periods}
  2K[\tfrac14]=3.3715007\ldots\quad(\DD),
  \qquad 2K[\tfrac34]=4.3130313\ldots\quad(\PP),
\end{equation}
so the ratio of the edges of the circumscribed cubes is
\begin{equation}\label{eq:schwarz}
  \frac{\lambda_{\PP}}{\lambda_{\DD}}
  =\frac{K[3/4]}{K[1/4]}=\frac{K'[1/4]}{K[1/4]}
  =1.2792615711710064662\ldots
\end{equation}
This is precisely the constant Schwarz obtained in $1866$ as the
condition for the adjoint hexagonal patches of $\DD$ and $\PP$ to have
equal area, hence for Bonnet bending of either surface into the other
\cite[vol.~I, p.~88]{Schwarz}; it also governs the area-to-volume
comparison of the two cells \cite{GKD}. So the ODE normalization
$|\kappa|=4$ with unit coefficients is \emph{automatically} the
isometric normalization of the conjugate pair --- an external check on
Theorem \ref{thm:unique} that uses no elliptic-function computation at
all.

\section{Equivalent forms, and the shape of the parallel}

\begin{proposition}[from ODE to equation]\label{prop:sep}
Let $\varphi$ be as in \eqref{eq:phis}. Passing to the
lattice-normalized variable of \eqref{eq:DC}:
\begin{enumerate}[label=\textup{(\alph*)},leftmargin=2.3em]
\item $\DD$: $\ \varphi(x)+\varphi(y)+\varphi(z)=0$ exponentiates to
      $\mathcal D(x)\mathcal D(y)\mathcal D(z)=1$;
\item $\PP$: with $\Psi$ as in \eqref{eq:Psi}, so that
$\tan\Psi=\sqrt3\,\mathcal C$, the relation
$\Psi(x)+\Psi(y)+\Psi(z)=0$ becomes, by the tangent addition
theorem, $\sum\mathcal C=3\prod\mathcal C$; quantitatively
      \begin{equation}\label{eq:err}
        \sin\bigl(\Psi(x)+\Psi(y)+\Psi(z)\bigr)
        =\Bigl[\cos\Psi(x)\cos\Psi(y)\cos\Psi(z)\Bigr]\cdot
         \sqrt3\ \mathcal P(x,y,z),
      \end{equation}
      the bracket being real-analytic and $\ge\frac18>0$ on
      $\RR^{3}$. Hence $\mathcal P$ and $\sin\sum\Psi$ agree in sign
      and in zero set, which is Theorem \ref{thm:P} and
      Remark \ref{rem:exc}.
\end{enumerate}
\end{proposition}

\begin{proposition}[the quadratic twins]\label{prop:quad}
Each equation has a M\"obius partner of shape\\
$\sum_{i<j}(\cdot)(\cdot)\pm1=0$, at a \emph{negative} parameter:
\begin{align}
  \DD:\quad &\mathcal F_{1}\mathcal F_{2}+\mathcal F_{2}\mathcal F_{3}
             +\mathcal F_{3}\mathcal F_{1}+1=0,
  &&\mathcal F(t)=\frac{\sn\dn}{\cn}
      \Bigl(K[-\tfrac13]\,t,\ -\tfrac13\Bigr),\label{eq:Dquad}\\
  \PP:\quad &\lambda_{1}\lambda_{2}+\lambda_{2}\lambda_{3}
             +\lambda_{3}\lambda_{1}-1=0,
  &&\lambda(t)=\sqrt3\,\sn^{2}\bigl(K[-3]\,t,\ -3\bigr).\label{eq:Pquad}
\end{align}
The link in the $\DD$ case is
\begin{equation}\label{eq:mobius}
  \mathcal F(t)=\frac{1+\mathcal D(t-\tfrac12)}{1-\mathcal D(t-\tfrac12)},
\end{equation}
which converts \eqref{eq:Deq} into \eqref{eq:Dquad}; in the $\PP$ case
$\lambda=\tan\bigl(\frac\pi6-\frac\Psi2\bigr)$ and \eqref{eq:Pquad} is
the statement $\sum\arctan\lambda_{i}=\frac\pi2$.
Equation \eqref{eq:Dquad} is the form obtained in \cite{Dpaper};
\eqref{eq:Pquad} is the additively separable form of
\cite{Ppaper}.
\end{proposition}

\begin{remark}[$-\tfrac13$ and $-3$]\label{rem:anh}
The two negative parameters in Proposition \ref{prop:quad} are
reciprocal, and they are the two roots of the palindromic quadratic
\begin{equation}\label{eq:palin}
  3m^{2}+10m+3=0,
\end{equation}
which is exactly the quadratic that the Taylor jets of the $\PP$
parametrisation force in \textup{\cite[\S11]{Ppaper}}. The reason is
\eqref{eq:mobius}: the function
$G=\frac{1+\mathcal E}{1-\mathcal E}$ built from
$\mathcal E'^{2}=1+\mathcal E^{2}+\mathcal E^{4}$ satisfies
$4G'^{2}=3G^{4}+10G^{2}+3$. Thus \eqref{eq:palin} is the discriminant
of the M\"obius map relating the product and quadratic shapes, and
$\{-3,-\frac13\}$ are the two negative members of the anharmonic orbit
$\{\frac14,\frac34,4,\frac43,-\frac13,-3\}$ of
$j=35152/9$, on which all the moduli of both surfaces lie.
\end{remark}

\section{Constants}

\begin{table}[ht]
\caption{The parallel. Rows above the rule are the mechanism; rows
below are consequences. $\varphi$ is normalized as in
Theorem \ref{thm:unique}.}\label{tab:par}
\begin{tabular}{@{}p{0.30\textwidth}p{0.31\textwidth}p{0.31\textwidth}@{}}
\toprule
 & $\DD$ (diamond) & $\PP$ (primitive)\\
\midrule
separation constant & $\kappa=+4$ & $\kappa=-4$\\
$\varphi'^{2}$ as a function of $\varphi$ & $1+2\cosh2\varphi$ &
  $1+2\cos2\varphi$\\
range of $\varphi$ & $\RR$ & $[-\frac\pi3,\frac\pi3]$\\
separating function &
  $\log\frac{\sn\dn}{\cn}(\cdot,\frac14)$ &
  $\am(\sqrt3\,\cdot,\frac43)
   =\arctan\bigl(\sqrt3\cn(2\,\cdot,\frac34)\bigr)$\\
addition theorem used & $\exp$ & $\tan$\\
\midrule
implicit equation & $\mathcal D_{1}\mathcal D_{2}\mathcal D_{3}=1$ &
  $\sum\mathcal C_{i}=3\prod\mathcal C_{i}$\\
quadratic twin &
  $\sum_{i<j}\mathcal F_{i}\mathcal F_{j}+1=0$, $m=-\frac13$ &
  $\sum_{i<j}\lambda_{i}\lambda_{j}-1=0$, $m=-3$\\
parameter & $\frac14$ & $\frac34=1-\frac14$\\
lattice / space group & fcc, $Fd\bar3m$ & bcc, $Im\bar3m$\\
cube edge (common norm.) & $2K[\frac14]$ & $2K[\frac34]$\\
patch spanned by & $4$ edges of a regular tetrahedron &
  $4$ edges of a regular octahedron\\
poles / absolute values & yes / needed over $\RR$ and the product form is valid only on a fundamental domain & none / none\\
exceptional set of the equation & none & one lattice orbit
  (Rem.\ \ref{rem:exc})\\
priority & Schwarz, Cayley, Nitsche \cite{Schwarz,Cayley,Nitsche} &
  family: Kim--Ogata \cite{KO}; equation: here\\
\bottomrule
\end{tabular}
\end{table}

\begin{table}[ht]
\caption{All entries derive from $K[1/4]$ and $K[1/9]$.}
\begin{tabular}{@{}lll@{}}
\toprule
quantity & value & role\\
\midrule
$K[1/4]$ & $1.6857503\ldots$ & $\DD$ quarter period; also 
   $p_{0}=\tfrac23K[8/9]$ \cite[\S84]{Nitsche}\\
$K[3/4]=K'[1/4]$ & $2.1565156\ldots$ & $\PP$ quarter period\\
$K[-3]=\frac12K[3/4]=\frac23K[1/9]$ & $1.0782578\ldots$ &
  argument scale in \eqref{eq:Pquad}\\
$K[-1/3]=\frac{\sqrt3}{2}K[1/4]$ & $1.4599026\ldots$ &
  argument scale in \eqref{eq:Dquad}\\
$F[\pi/3\,|\,4/3]=\sqrt3\,K[-3]$ & $1.8675973\ldots$ &
  constant of \cite{KO}; a complete integral in disguise\\
$K[3/4]/K[1/4]$ & $1.2792615\ldots$ &
  Schwarz's scale ratio \eqref{eq:schwarz}\\
$j=35152/9$ & $3905.7777\ldots$ &
  the one elliptic curve (Rem.\ \ref{rem:anh})\\
\bottomrule
\end{tabular}
\end{table}

\section{Method}\label{sec:method}

Proofs are deferred to \cite{Dpaper,Ppaper}; the four steps for
Theorem \ref{thm:P} are short and we indicate them.

\begin{enumerate}[label=\textup{(\arabic*)},leftmargin=2.4em]
\item \emph{The constants.} $F[\pi/3\,|\,4/3]=\frac{\sqrt3}{2}K[3/4]
      =\sqrt3\,K[-3]$: an incomplete integral that is complete in
      disguise, the endpoint $\pi/3$ being exactly $\arcsin(1/\sqrt m)$.
      This is what makes the three amplitudes of the Kim--Ogata
      equation collapse.
\item \emph{Two transformation laws.} Reciprocal parameter
      ($m\mapsto1/m$, taking $\tfrac43$ to $\tfrac34$) and Gauss
      ($m=-3\mapsto\tfrac34$) reduce all three amplitudes to the single
      bounded function $\Psi$ of \eqref{eq:Psi}, with
      $\sin\Psi=\frac{\sqrt3}{2}\cn$, $\cos\Psi=\frac12\dn$ at
      parameter $-3$; the duplication formula for $\cn$ then gives
      $\Psi=\frac\pi3-2\arctan(\sqrt3\sn^{2})$ and a Landen step gives
      the $\tfrac19$ dictionary of \cite{Ppaper}.
\item \emph{Minimality and regularity} of
      $\{\Psi(x)+\Psi(y)+\Psi(z)=0\}$, by the three-line trigonometric
      identity that \eqref{eq:mse} becomes under
      $s_{i}^{2}=\frac13(1+2\cos2\Psi_{i})$; the gradient
      $\propto(\sn(2\cdot),\sn(2\cdot),\sn(2\cdot))$ vanishes only at
      points of $\Lat$, which are not on the surface. This supplies the
      converse direction that \cite[Prop.~4.1]{KO} states but does not
      prove (that paper's Proposition is the implication
      ``$\Sigma$ separable minimal $\Rightarrow$ $\Sigma$ has the
      amplitude form'').
\item \emph{Identification} with the Weierstrass patch of $\PP$
      (Gauss map $G=\zeta$, height differential
      $2\zeta\dd\zeta/\sqrt{1+14\zeta^{4}+\zeta^{8}}$) by Bj\"orling's
      theorem along the straight segment
      $\{(t,-t,\frac12)\}$, on which the two unit normal fields are
      shown to coincide identically. This is the method Nitsche uses in
      \cite[\S85]{Nitsche} to identify \emph{his} two representations
      of $\DD$ along the line $y=z=0$; we simply reuse it.
\end{enumerate}
Steps (1)--(4) also yield the modulus-$(-3)$ form
$\sn(\varpi(x{+}y),-3)\sn(\varpi(x{-}y),-3)
 =\frac13\sn(3\varpi(z-\frac12),\frac19)$, $\varpi=K[-3]$, which is the
form the Weierstrass data produce and the main theorem of
\cite{Ppaper}, together with the local converse proved there.

\section{Quartic Integrals}\label{sec:add}

The identification \eqref{eq:NitscheE} is one instance of a
one-parameter family of quartic integrals, which we record because it
supplies, in a single stroke, Nitsche's constant, the period of
$\PP$, and the normalizing constant of \cite{Ppaper}.

\begin{lemma}[integration]\label{lem:quartic}
	For $m<1$ put $T_{m}:=\dfrac{\sn\dn}{\cn}(\cdot\,,m)$. Then
	\begin{equation}\label{eq:Tode}
	T_{m}'=1-2m\sn^{2}+T_{m}^{2}>0,
	\qquad\text{hence}\qquad
	T_{m}'^{\,2}=1+(2-4m)\,T_{m}^{2}+T_{m}^{4},
	\end{equation}
	and $T_{m}$ is an increasing bijection from $[0,K[m])$ onto
	$[0,\infty)$. Consequently
	\begin{equation}\label{eq:quartic}
	\boxed{\ \int_{0}^{\infty}
		\frac{\dd\tau}{\sqrt{\tau^{4}+(2-4m)\,\tau^{2}+1}}=K[m]
		\qquad (m<1).\ }
	\end{equation}
\end{lemma}

\begin{proof}
	Write $S=\sn(u,m)$, $C=\cn$, $D=\dn$ and $s=S^{2}$, so that
	$C^{2}=1-s$, $D^{2}=1-ms$ and $T_{m}^{2}=s(1-ms)/(1-s)$. From
	$S'=CD$, $C'=-SD$, $D'=-mSC$,
	\[
	T_{m}'=\frac{CD\cdot D+S(-mSC)}{C}+\frac{SD\cdot SD}{C^{2}}
	=\bigl(D^{2}-mS^{2}\bigr)+T_{m}^{2}
	=1-2ms+T_{m}^{2},
	\]
	which is the first part of \eqref{eq:Tode}; clearing denominators,
	\[
	T_{m}'=\frac{(1-2ms)(1-s)+s(1-ms)}{1-s}
	=\frac{1-2ms+ms^{2}}{1-s}>0,
	\]
	the numerator being positive for every $s\in[0,1]$ because its
	discriminant in $s$ is $4m(m-1)<0$ when $0<m<1$, while for $m\le0$ one
	has $ms(s-2)\ge0$. Squaring $T_{m}'=1-2ms+T_{m}^{2}$ and subtracting
	the asserted quartic leaves
	$4m\bigl[ms^{2}-s+(1-s)T_{m}^{2}\bigr]=0$, which holds identically by
	the closed form of $T_{m}^{2}$; this is the second part of
	\eqref{eq:Tode}. Since $m<1$ we have $D\ge\sqrt{1-m}>0$ on
	$[0,K[m]]$, so $T_{m}\to+\infty$ as $u\uparrow K[m]$, where $C$ has its
	first zero; with $T_{m}(0)=0$ and $T_{m}'>0$ this gives the asserted
	bijection. Finally $2-4m>-2$, so $\tau^{4}+(2-4m)\tau^{2}+1>0$ for real
	$\tau$, and \eqref{eq:Tode} may be written
	$\dd u=\dd T_{m}\bigl/\sqrt{T_{m}^{4}+(2-4m)T_{m}^{2}+1}$; integrating
	over $u\in[0,K[m])$ yields \eqref{eq:quartic}.
	
	Alternatively, and without elliptic functions: the substitution
	$\nu=\tau-\tau^{-1}$ gives
	$\tau^{4}+(2-4m)\tau^{2}+1=\tau^{2}\bigl(\nu^{2}+4-4m\bigr)$ and
	$\dd\tau/\tau=\dd\nu/\sqrt{\nu^{2}+4}$, whence the left side of
	\eqref{eq:quartic} equals
	\[
	\int_{-\infty}^{\infty}
	\frac{\dd\nu}{\sqrt{(\nu^{2}+4-4m)(\nu^{2}+4)}}=K[m],
	\]
	by $\int_{0}^{\infty}\dd\nu\bigl/\sqrt{(\nu^{2}+a^{2})(\nu^{2}+b^{2})}
	=b^{-1}K[1-a^{2}/b^{2}]$ for $0<a\le b$ (substitute
	$\nu=a\tan\vartheta$), together with, when $m<0$, the 
	relation $K[m]=(1-m)^{-1/2}K[-m/(1-m)]$.
\end{proof}

\begin{table}[ht]
	\caption{Special cases of Lemma \ref{lem:quartic}. The two cube edges
		of Table \ref{tab:par} are the two signs of the same
		integral.}\label{tab:quartic}
	\begin{tabular}{@{}llll@{}}
		\toprule
		$m$ & $2-4m$ & $K[m]$ & occurrence\\
		\midrule
		$0$ & $2$ & $\pi/2$ & degenerate check\\
		$\tfrac14$ & $1$ & $1.6857503\ldots$ &
		Nitsche's $p_{0}=\tfrac23K[\tfrac89]$
		\textup{\cite[\S84]{Nitsche}}; half the $\DD$ cube edge\\
		$\tfrac12$ & $0$ & $\tfrac14B(\tfrac14,\tfrac14)=1.8540746\ldots$ &
		classical lemniscatic case\\
		$\tfrac34$ & $-1$ & $2.1565156\ldots$ &
		half the $\PP$ cube edge \eqref{eq:periods}\\
		$-3$ & $14$ & $1.0782578\ldots$ &
		$\varpi=K[-3]$, the normalising constant of \cite{Ppaper}\\
		\bottomrule
	\end{tabular}
\end{table}

\begin{remark}\label{rem:quartic}
	Three consequences. \textup{(i)} At $m=\tfrac14$ the quartic is
	$\tau^{4}+\tau^{2}+1$ and \eqref{eq:quartic} is Nitsche's evaluation,
	so the argument $K[\tfrac14]\,t$ in \eqref{eq:DC} is exactly his
	$p_{0}t$ and the cube edge $2K[\tfrac14]$ is his cube of side
	$2p_{0}$. \textup{(ii)} At $m=\tfrac34$ the quartic is
	$\tau^{4}-\tau^{2}+1$: the $\DD$ and $\PP$ periods are the two signs of
	one integral, which is the arithmetic shadow of the $\cosh/\cos$
	dichotomy of Theorem \ref{thm:unique}. \textup{(iii)} At $m=-3$ the
	coefficient is $2-4(-3)=14$ --- the same $14$ as in the Weierstrass
	polynomial $1+14\zeta^{4}+\zeta^{8}$ shared by the two surfaces
	\textup{\cite[\S85]{Nitsche}}, and \eqref{eq:quartic} then reads
	$\int_{0}^{\infty}(1+14\tau^{2}+\tau^{4})^{-1/2}\dd\tau=\varpi$, which
	is how $\varpi$ arises in \cite{Ppaper}.  See related work in \cite{Ross}.
\end{remark}

\section{Credit, and what is claimed}\label{sec:credit}

Separable minimal surfaces were introduced by Weingarten \cite{Wein};
Cayley \cite{Cayley} determined the class defined by
\eqref{eq:Dquad}-type equations, and Fr\'echet \cite{Frechet} returned
to the problem. Nitsche's \emph{Lectures}
\cite[\S\S82--86]{Nitsche} contain the reduction \eqref{eq:mse}, the
separation constant $\kappa$, the diamond equation \eqref{eq:Deq}
together with its M\"obius twin, the correct Weierstrass data
$P(\omega)=-(1+14\omega^{4}+\omega^{8})$, and the identification of the
implicit and parametric descriptions of $\DD$. \textbf{Theorem
\ref{thm:D} is therefore classical, and priority for the diamond
equation belongs to Schwarz, Cayley and Nitsche.} Nitsche set up the
case $\kappa<0$ but worked out no example in it; every surface exhibited
in \cite[\S\S83--86]{Nitsche} has $\kappa=4$ or $\kappa=0$. Kim \&
Ogata \cite{KO} classified all separable minimal surfaces and were, as far
as we know, the first to treat $\kappa<0$, obtaining the two-parameter
amplitude family and asserting that it contains $\PP$.

Against that background we claim: the closed form
\eqref{eq:NitscheE} of Nitsche's function (the constant 
$p_0=\frac23K[8/9]$ appears in \textup{\cite[\S84]{Nitsche}});
the trilinear equation \eqref{eq:Peq} for $\PP$ with its exceptional
set, the specification of the branch of $\am$ at parameter $>1$
(without which the amplitude equation is ambiguous, and false in the
naive reading), and the exact constants; the identification of the
member of \cite{KO} that is $\PP$, hence a proof of
\cite[Thm.~1(3)]{KO}; the uniqueness statement
Theorem \ref{thm:unique} with its $\cosh$/$\cos$ dichotomy; the
emergence of Schwarz's constant \eqref{eq:schwarz} from the common
normalization; and the bridge from both equations to the Weierstrass
data, carried out in \cite{Dpaper,Ppaper}. Exact but
\emph{parametric} descriptions of $\DD$, $\PP$ and the gyroid are due
to Gandy, Klinowski and coauthors \cite{GKD,GKP,GKG}.

\begin{problem}
The nodal approximations $\cos\pi x+\cos\pi y+\cos\pi z=0$ of $\PP$ and
$\sin\pi x\sin\pi y\sin\pi z+\cdots=0$ of $\DD$ are the degenerate
($m\to0$) models of \eqref{eq:Peq} and \eqref{eq:Deq}. Compare them
harmonic by harmonic; $\mathcal C$ is an even $2$-antiperiodic function
whose Fourier series begins with $\cos\pi t$.
\end{problem}

\begin{problem}
Which members of the associate family joining $\DD$ to $\PP$ admit
exact implicit representations? None can be separable, by
Theorem \ref{thm:unique}; the gyroid, at Bonnet angle
$38.0147\ldots^{\circ}$ \cite{GKG}, is the natural test case, and a
genuinely different mechanism is required.
\end{problem}

\begin{problem}
	Both equations are \emph{trilinear}: the variables 
	$\mathcal D_{i}$ and $\mathcal C_{i}$ each have degree at most $1$ in 
	each term. Is trilinearity forced? Precisely:
	if $\Sigma\subset\RR^3$ is a triply periodic minimal surface admitting
	an implicit equation $\sum_{\varepsilon\in\{0,1\}^3}
	c_\varepsilon\,F(x)^{\varepsilon_1}F(y)^{\varepsilon_2}
	F(z)^{\varepsilon_3}=0$ for one elliptic $F$ and constants
	$c_\varepsilon$, must $\Sigma$ be $D$ or $P$? Theorem \ref{thm:unique} 
	answers this for the separable subclass, where the only trilinear 
	shapes are $\prod \mathcal D_{i}=1$ and 
	$\sum \mathcal C_{i}=3\prod \mathcal C_{i}$ together with their
	M\"obius partners \eqref{eq:Dquad} and \eqref{eq:Pquad}. Might 
	Theorem \ref{thm:unique} be true for the (potentially larger) 
	trilinear subclass?
\end{problem}

\begin{problem}
	The two lattice periods $2K[\frac14]$, $2K[\frac34]$ are the two signs
	of one quartic integral (Table \ref{tab:quartic}), and the anharmonic 
	orbit of $j=35152/9$ carries all six moduli of both surfaces. Does the 
	elliptic curve $j=35152/9$ have a modular interpretation accounting 
	for the $14$ in $1+14\zeta^4+\zeta^8$ --- equivalently, for $2-4m=14$ 
	at $m=-3$ (Remark \ref{rem:quartic}(iii))?
\end{problem}

\subsection*{Acknowledgements}
I thank Yuta Ogata for correspondence and for a clarifying note on
\cite{KO}. Anthropic Claude Opus 5.0 played a crucial role in turning 
my conjectures into theorems: to say that the AI model merely ``assisted'' 
would be an understatement.  My interest in minimal surfaces 
began years ago (see \cite{Finch}).

\pagebreak


\begin{thebibliography}{99}
\bibitem{BF} P.\,F. Byrd and M.\,D. Friedman, \emph{Handbook of
  Elliptic Integrals for Engineers and Scientists}, 2nd ed., Springer,
  1971.
\bibitem{Cayley} A. Cayley, \emph{On a special surface of minimum
  area}, Quart.\ J.\ Pure Appl.\ Math.\ \textbf{14} (1877) 190--196.
\bibitem{Frechet} M. Fr\'echet, \emph{D\'etermination des surfaces
  minima du type $a(x)+b(y)=c(z)$}, Rend.\ Circ.\ Mat.\ Palermo
  \textbf{5} (1956) 238--259; \textbf{6} (1957) 5--32.
\bibitem{GKD} P.\,J.\,F. Gandy, D. Cvijovi\'c, A.\,L. Mackay and
  J. Klinowski, \emph{Exact computation of the triply periodic D
  (``diamond'') minimal surface}, Chem.\ Phys.\ Lett.\ \textbf{314}
  (1999) 543--551.
\bibitem{GKP} P.\,J.\,F. Gandy and J. Klinowski, \emph{Exact
  computation of the triply periodic Schwarz P minimal surface},
  Chem.\ Phys.\ Lett.\ \textbf{322} (2000) 579--586.
\bibitem{GKG} P.\,J.\,F. Gandy and J. Klinowski, \emph{Exact
  computation of the triply periodic G (``gyroid'') minimal surface},
  Chem.\ Phys.\ Lett.\ \textbf{321} (2000) 363--371.
\bibitem{HL} T. Hasanis and R. L\'opez, \emph{Classification of
  separable surfaces with constant Gaussian curvature},
  Manuscripta Math.\ \textbf{166} (2021) 403--417.
\bibitem{KO} D. Kim and Y. Ogata, \emph{Separable minimal surfaces and
  their limit behavior}, J. Korean Math.\ Soc.\ \textbf{61} (2024),
  no.~4, 761--778.
\bibitem{Nitsche} J.\,C.\,C. Nitsche, \emph{Lectures on Minimal
  Surfaces}, Vol.~1, Cambridge Univ.\ Press, 1989, \S\S82--86.
\bibitem{Ross} M. Ross, \emph{$\PP$ and $\DD$ surfaces are stable},    
Differential Geom. Appl.\ \textbf{2} (1992), no.~2, 179--195.
\bibitem{Schwarz} H.\,A. Schwarz, \emph{Gesammelte Mathematische
  Abhandlungen}, Band I, Springer, 1890 (pp.~1--125; p.~88 for
  \eqref{eq:schwarz}).
\bibitem{Wein} J. Weingarten, \emph{Ueber die durch eine Gleichung von
  der Form $x+y+z=0$ darstellbaren Minimalfl\"achen}, G\"ott.\ Nachr.\
  \textbf{1887} (1887) 272--275.
\bibitem{Dpaper} S. R. Finch, \emph{Elimination of the parameters in
  an elliptic parametrisation of the Schwarz diamond surface}, in
  preparation. 
\bibitem{Ppaper} S. R. Finch, \emph{Elimination of the parameters in
  an elliptic parametrisation of the Schwarz primitive surface}, in
  preparation. 
\bibitem{Finch} S. R. Finch, \emph{Mathematical Constants II}, Cambridge Univ.\ 
Press, 2019; \S5.17, Gergonne-Schwarz surface; \S5.18, Partitioning problem.
\end{thebibliography}
\end{document}